\documentclass[a4paper,10pt]{article}
\usepackage{amsmath}
\usepackage{amssymb}
\usepackage{amsthm}
\usepackage{amsfonts}
\usepackage{mathdots}
\usepackage{enumerate}
\usepackage{comment}

\title{On Forbidden Submatrices}
\author{Ar\`es M\'eroueh\footnote{Department of Pure Mathematics and Mathematical Statistics, Centre for Mathematical Sciences, Wilberforce Road, Cambridge CB3 0WB, United Kingdom; e-mail: {\tt a.j.meroueh@dpmms.cam.ac.uk}}}

\newcommand{\C}{\mathcal{C}}
\newcommand{\A}{\mathcal{A}}
\newcommand{\F}{\mathcal{F}}
\newcommand{\Cia}{\C_i(\A)}
\newcommand{\tA}{\tilde{\mathcal{A}}}
\newcommand{\PO}{\mathcal{P}}
\newcommand{\N}{\mathbb{N}}
\newcommand{\fs}{\mathrm{fs}}
\newcommand{\supp}{\mathrm{supp}}
\begin{document}

\date{}
\maketitle
\begin{abstract}
Given a $k\times l$ $(0,1)$-matrix $F$, we denote  by $\mathrm{fs}(m,F)$ the largest number for which there is an $m \times \mathrm{fs}(m,F)$  $(0,1)$-matrix with no repeated columns and no induced submatrix equal to $F$. A conjecture of Anstee, Frankl, F\"{u}redi and Pach states that $\mathrm{fs}(m,F) = O(m^k)$ for a fixed matrix $F$. The main results of this paper are that $\mathrm{fs}(m,F) = m^{2+ o(1)}$  if $k=2$ and  that $\mathrm{fs}(m,F) = m^{5k/3 -1 + o(1)}$ if $k\geq 3$.
\end{abstract}

\section{Introduction}
 How large can a matrix of zeros and ones be if it does not contain a given matrix $F$ as a submatrix? There are essentially two different ways in which this question can be interpreted, depending on what we mean by \emph{containing} $F$. A common approach is to say that a matrix $M$ \emph{contains} a matrix $F$ if $F$ can be obtained from $M$ by deleting some rows and columns of $M$ and then possibly turning some ones into zeros. In this setting we may view the ones of $M$ as representing edges of a bipartite graph $G$ so that forbidding $F$ in $M$ amounts to forbidding an ordered bipartite subgraph of $G$. Keeping this analogy in mind, we can restate the general extremal question in a more rigorous way: what is the maximal number of ones in an $n\times m$ matrix $M$ which does not contain $F$? 

The problem which we consider here, however, has a different flavour. We will say that a matrix $M$ \emph{contains} a matrix $F$ if $F$ can be obtained from $M$ by deleting some rows and columns of $M$. That is, we do not allow ourselves to turn some ones into zeros. In other words, we forbid $F$ as an \emph{induced} submatrix of $M$. Asking for the maximal number of ones in a matrix $M$ not containing $F$ no longer makes sense in this context. Rather, we make the following definition. 

\newtheorem{dfs}{Definition}[section]
\begin{dfs}
A $(0,1)$ matrix is said to be \emph{simple} if its columns are pairwise distinct. We define $\fs(m,F)$ to be the maximal number of columns of a simple matrix $M$ on $m$ rows which does not contain $F$. 
\label{dfs}
\end{dfs}

Notice that $\fs(m,F)$ is always defined and in fact is no larger than $2^m$. The problem of determining $\fs(m,F)$ for a given matrix $F$ was first raised in \cite{ffp} and \cite{ansteefuredi} where Anstee,  Frankl, F\"{u}redi and Pach  made the following conjecture.

\newtheorem{conjecture}[dfs]{Conjecture }
\begin{conjecture}[Frankl, F\"{u}redi and Pach \cite{ffp}, Anstee and F\"{u}redi \cite{ansteefuredi}]
\label{conjecture}
Let $F$ be a $k\times l$ $(0,1)$ matrix. There exists a constant $c_F$ such that $$\fs(m,F) \leq c_F m^{k}. $$ 
\end{conjecture}

Frankl, F\"{u}redi and Pach \cite{ffp} showed that $$\fs(m,F) = O_F(m^{2k-1}).$$ This was later improved by Anstee \cite{ansteebound} to $$\fs(m,F) = O_F(m^{2k-1 - \epsilon })$$ where $\epsilon = (k-1)/(13\log_2 l)$. While this is a significant improvement for small $l$, for large $l$ this is approximately $m^{2k-1}$ so that the gap between the conjectured bound and the current best bound is quite large. 

The problem was further studied for various fixed values of $k$ and $l$, as well as for some small matrices $F$ in \cite{ansteefuredi} by Anstee and F\"{u}redi. More recently in  \cite{ansteechen} the conjecture was shown to hold for some families of $2\times l$ matrices by Anstee and Chen. However, to date the conjecture is still unknown for $k\geq 2$. In this paper we aim to prove the following new upper bound for any forbidden submatrix $F$.

\newtheorem{maintheorem}[dfs]{Theorem}
\begin{maintheorem}
\label{maintheorem}
For any fixed $2\times l$ matrix $F$, $$\fs(m,F) = m^{2+ o(1)}. $$
For $k\geq 3$ and any fixed $k\times l$ matrix $F$, $$ \fs(m,F) = m^{5k/3 -1 +  o(1)}.$$ 
\end{maintheorem}

This gives further evidence for Conjecture \ref{conjecture}. For $k\geq 4$ the bound given by the theorem may be improved for some fixed values of $k$, but only by a small constant term in the exponent (see Section 6). However some new ideas would be required to bring the coefficient of $k$ in the exponent  down to a constant smaller than $5/3$. 
\section{Contributions}

Given two integers $a,b$ such that $a \geq b$, we denote by $[a]^{(b)}$ the set of subsets of $\{1,2, \ldots, a \}$ of size $b$, and by  $[a]^{(\geq b)}$ the set of subsets of $\{1,2, \ldots, a \}$ of size at least $b$. Let $M$ be an $m \times n$ simple matrix. Given $s \leq m$, $t \leq n$, and $R \in [m]^{(s)}$, $C \in [n]^{(t)}$ we denote by $M[R,C]$ the $s\times t$ submatrix of $M$ whose rows are indexed by $R$ and columns by $C$. The next definition generalizes the concept of \emph{contributions} introduced by Anstee and Chen in \cite{ansteechen}. 

\newtheorem{contribs}{Definition}[section]
\begin{contribs}
\label{contribs}
A collection $(R_1,C_1),(R_2,C_2),\ldots, (R_i,C_i) \in [m]^{(k)}\times [n]^{(2^k)}$ is said to form a set of $i$ \emph{contributions} if $M[R_j,C_j]$ is a simple matrix for all $j\leq i$ and also for any $j \neq j'$, if $R_j = R_{j'}$ then either $\max C_j < \min C_{j'}$ or $\max C_{j'} < \min C_{j} $. If there exists a set of at least $i$ contributions for $M$ then we say that $M$ makes $i$ contributions. 
\end{contribs}

In other words, a matrix makes $i$ contributions if there exist $i$ simple submatrices on $k$ columns and $2^k$ rows (with repeats allowed), such that two submatrices appearing on the same set of rows are such that the last column of one appears before the first column of the other. We also call each such submatrix \emph{a contribution}. 

The following elementary lemma motivates the introduction of contributions.
\newtheorem{basiclemma}[contribs]{Lemma}
\begin{basiclemma}
\label{basiclemma}
Let $0\leq a \leq k-1$. Suppose that there exists $c>0$ such that every simple $m\times \lfloor cm^{k-1} \rfloor$ matrix makes $m^a$ contributions. Then for any $k\times l$ matrix $F$, $$\fs(m,F) = O_F(m^{2k-1-a}). $$
\end{basiclemma}

\begin{proof}
 It is easy to see that the concatenation of two matrices $M_1$ and $M_2$ making $i_1$ and $i_2$ contributions respectively makes at least $i_1 + i_2$ contributions. Therefore a simple $\lceil (l/k!) m^{k-a} \rceil \lfloor c m^{k-1} \rfloor $ matrix makes at least $l{ m \choose k}$ contributions. Since there are $ {m \choose k}$ possible row indices for these contributions, at least $l$ of them must share the same row index. As the corresponding simple $k\times 2^k$ submatrices of $M$ appear consecutively and each contains all possible columns on $k$ rows, we see that $M$ contains all possible $k\times l$ matrices $F$. This shows that $\fs(m,F) = O(m^{2k-1-a})$, as required.
\end{proof}

In order to to apply Lemma \ref{basiclemma} we need to be able to find contributions in simple matrices. The next proposition, due to Sauer \cite{sauer} and Shelah \cite{shelah}, will prove to be an essential tool in doing so.

\newtheorem{shattered}[contribs]{Propostion}
\begin{shattered}[Sauer \cite{sauer}, Shelah \cite{shelah}]
\label{shattered}
Let $\A \subseteq \mathcal{P}[m]$. If $|\A| \geq { m \choose 0} + { m \choose 1 } + \cdots + {m \choose k-1} + 1$ then there exists $S\in [m]^{(k)}$ such that $|S\cap \A| = 2^k$, where $S\cap \A = \{S\cap A:\, A\in \A\}$. We say that $\A$ \emph{shatters} $S$.
\end{shattered}

An $m\times 1$ column of zeros and ones can be viewed as representing a subset of $[m]$, if we think of this column as an indicator function. Likewise, we may view a simple matrix $M$ as representing the family $\A$ of subsets of $[m]$ where each set of $\A$ corresponds to a column of $M$.  We call $\A$ the family \emph{associated} with $M$. A key observation  is that  if $\A$ shatters a subset $S$ of $[m]$ then there is a simple $|S|\times 2^{|S|}$ submatrix of $M$ whose rows are indexed by $S$. Thus if we consider the family $\A$ associated with a simple $m \times \left( { m \choose 0} + { m \choose 1 } + \cdots + {m \choose k-1} + 1\right)$ matrix $M$ we find by Proposition \ref{shattered} at least one subset of $[m]$ of size $k$ shattered by $\A$ and hence there is at least one contribution in $M$. Therefore any simple $m \times \Omega(m^{k-1})$ simple matrix makes at least one contribution and by Lemma \ref{basiclemma} we have that $\fs(m,F) = O(m^{2k-1})$, thus proving the polynomial upper bound on $\fs(m,F)$ due to Frankl, F\"{u}redi and Pach \cite{ffp}. 

The natural question which arises from this argument is whether one can find substantially more contributions from $\Omega({m^{k-1}})$ pairwise distinct columns. We answer this question in the affirmative: in the next two sections our aim is to show that for  $k=2$ a simple  $m\times \Omega(m)$ matrix makes at least $m^{1-o(1)}$ contributions while for $k\geq 3$ an $m \times \Omega(m^{k-1})$ simple matrix makes at least $m^{k/3 - o(1)}$ contributions. However, we believe that a much stronger result should hold, which would imply Conjecture \ref{conjecture}.

\newtheorem{conjecture2}[contribs]{Conjecture}
\begin{conjecture2}
\label{conjecture2}
There exists a constant $c >0$ such that any simple $m\times \lfloor c m^{k-1} \rfloor$ matrix makes $m^{k-1}$ contributions. 
\end{conjecture2}

\section{$2 \times l$ matrices}

Bearing in mind that a simple matrix $M$ has an associated family $\A$, we begin by defining a useful operation on families of sets called \emph{compression}.

\newtheorem{compressions}{Definition}[section]
\begin{compressions}
Let $m\in \N$ and $i\in [m]$. We call \emph{compressions} the two maps $C_i$ and $\mathcal{C}_i$ defined as follows.

For $A\subseteq [m]$,
$$C_i(A) = \left\{
	\begin{array}{ll}
		A  & \mbox{if } i \not \in A \\
		A \backslash \{i\} & \mbox{if } i \in A.
	\end{array}
\right. $$
For $\mathcal{A} \subseteq \mathcal{P}[m]$,
$$\mathcal{C}_i(\mathcal{A}) = \left\{ C_i(A): A\in \mathcal{A}\right\} \cup \left\{ A: A \in \mathcal{A} \text{ and }  C_i(A) \in \mathcal{A}\right\}. $$

\label{compressions}
\end{compressions}
Recall from Lemma \ref{shattered} that given a family $\A$ of subsets of $[m]$ and a set $S \subseteq [m]$, we denote by $S\cap \A$ the set $ \{S\cap A:\, A\in \A\}$. We next prove a simple but important lemma about compressions. 
\newtheorem{isets}[compressions]{Lemma}

\begin{isets}
Let $S\subseteq [m]$, $i\in [m]$ and $\mathcal{A} \subseteq \mathcal{P}([m])$. Then
$$|S \cap \A| \geq |S \cap \C_i(\A)|. $$

\label{isets}
\end{isets}
\begin{proof}
Define the map $\psi :S \cap \Cia  \, \backslash\,  S\cap \A  \longrightarrow S \cap \A \, \backslash\, S \cap \Cia$ by $\psi(B) = B \cup\{i\}$. We will check that $\psi$ is injective and well-defined and this proves the lemma. Let $B \in S \cap \Cia  \, \backslash\,  S\cap \A$. Then $B = S \cap (A \backslash \{i\})$ for some $A\in \A$ with $C_i(A) \not \in \A$ (so that $i\in A$). Thus $i \not \in B$ and hence $\psi$ is injective. All that remains is to show that $\psi$ is well-defined, that is $\psi(B) \in  S \cap \A \, \backslash\, S \cap \Cia$. Certainly $B \cup \{i \} = S \cap A$ so $B \in S \cap \A$. Suppose for a contradiction that $B \cup \{i \} \in S \cap \Cia$. Then $B \cup \{i\} = S \cap A'$ where $A' \in \Cia$ and $i\in A'$,  hence $A' \backslash \{i \} \in \A$ meaning $B = S \cap (A' \backslash \{i\}) \in S \cap \A$, contradicting $B \not \in S\cap \A$. 
\end{proof}

\newtheorem{isetsplus}[compressions]{Corollary}

\begin{isetsplus}
Let $\A$ be a family of subsets of $[m]$. There exists a family $\tA$ of subsets of $[m]$ such that for any $A \in \tA$ $$\PO(A) \subseteq \tA $$ and moreover for any $S \subseteq [m]$ $$ |S\cap \A| \geq |S \cap \tA|.$$
\label{isetsplus}
\end{isetsplus}

\begin{proof}
Start with $\A_0 = \A$ and thereafter if there exists $i$ such that $\A_j \neq \C_i(\A_j)$ we let $A_{j+1} =  C_i(\A_j)$. If there is no such $i$ then the process stops and we let $\tA$ be the last family obtained by this process. It is clear that this process does eventually stop since if  $\A_j \neq \C_i(\A_j)$ then $\sum_{A\in A_{j+1}} |A| < \sum_{A\in \A_j}|A|$. The fact that $\PO(A) \subseteq \A$ for any $A \in \tA$ follows from the fact that $\tA$ is stable under taking compressions, and the fact that $ |S\cap \A| \geq |S \cap \tA|$ for any $S \subseteq [m]$ from Lemma \ref{isets}.
\end{proof}

A family $\A$ satisfying the condition $\PO(A) \subseteq \A$ for any $A\in \A$ is called a \emph{down} family.

The utility of Lemma \ref{isets} now becomes apparent: for example, suppose one wishes to prove that a simple matrix $M$ makes a contribution. One could consider the associated family $\A$ and then the family $\tA$ given by Corollary \ref{isetsplus}. If $\tA$ contains a set $X$ of size $k$ then, as it is a down family, we have $|X \cap \tA| \geq 2^k$. Then by Corollary \ref{isetsplus} $|X \cap \A| \geq 2^k$, so that the submatrix of $M$ whose rows are indexed by $X$ makes at least one contribution. This is in fact one way in which Proposition \ref{shattered} may be proved.

It turns out that it is possible to find many more contributions in a simple $m\times \Omega(m^{k-1})$ matrix by using Corollary \ref{isetsplus}. The case $k=2$ is simplest and is treated in the rest of this section (although Lemma \ref{mstep} below is stated for general $k$, we will only apply it for $k=2$). For $k\geq 3$ a slightly more involved argument is given in Section 4. 

\newtheorem{mstep}[compressions]{Lemma}
\begin{mstep}
Let $k\geq 2$. Suppose that there exist $c > 0$ and $0 < \gamma \leq 1$ such that, for all $m\geq k$, every $m \times \lfloor c m^{k-1} \rfloor$ simple matrix $M$ makes $m^{1 - \gamma}$ contributions. Then there exists $c' > 0$ such that every $\lfloor m \times c' m^{k-1} \rfloor$ simple matrix $M$ makes $m^{1 - \gamma'}$ contributions, where $\gamma'= \frac{\gamma}{k-1+\gamma}$. 
\label{mstep}
\end{mstep}

Before proving Lemma \ref{mstep} we make an important definition. The \emph{support} of a family of sets $\A \subseteq \PO [m]$, denoted by $\supp(A)$, is defined to be $$\supp(A) = \{ x\in [m]: \exists A \in \A \mbox{ s.t. } x \in A\} $$

\begin{proof}[Proof of Lemma \ref{mstep}]

Let $M$ be an $m \times \lfloor c' m^{k-1} \rfloor$ simple matrix where $$c' = 2+ c''$$ and $$c'' = 2ck^{k+\gamma-2}.$$ Let $\A$ be the family associated to $M$ and let $\tA$ be the family given by Corollary \ref{isetsplus} applied to $\A$. Also let $X = \supp(\tA^{(\geq k)})$. Since $$\lfloor c'm^{k-1} \rfloor \geq  { m \choose 0} + { m \choose 1 } + \cdots + {m \choose k-1} + c''m^{k-1}$$we have $|X \cap \tA^{(\geq k)}| \geq c'' m^{k-1}$. Therefore, trivially, $|X \cap \tA| \geq c''m^{k-1}$ and so $|X \cap \A| \geq c'' m^{k-1}$ by Corollary \ref{isetsplus}. 

Assume first that $|X| \geq km^{1-\gamma'}$. Then clearly $\tA$ contains at least $m^{1 
- \gamma'}$ sets of size $k$ (indeed, by definition of $X$, each $x\in X$ belongs to at 
least one element of $\tA$ of size at least $k$ and as $\tA$ is a down family, $x$ 
belongs to at least one element of size exactly $k$. However, a set of size $k$ contains 
at most $k$ elements of $X$, hence $|\tA^{(k)}| \geq |X|/k$). As $\tA$ is a down family each element $\tA^{(k)}$ is shattered by $\tA$ and by Corollary \ref{isetsplus}  $\A$ also shatters these sets. Therefore $\A$ shatters at least $m^{1-\gamma'}$ sets of size $k$, and so $M$ makes at least $m^{1-\gamma'}$ 
contributions. 

Suppose now that $|X| \leq km^{1-\gamma'}$.  As $|X \cap \A| \geq c'' m^{k-1}$, there exists an $|X| \times \lceil c'' m^{k-1} \rceil$ simple submatrix of $M$ with row index $X$, call it $M_X$. Write $M_X = [B_1, B_2, \ldots, B_l]$, the concatenation of at least $\left\lfloor c''m^{k-1}/ \lfloor c_k|X|^{k-1} \rfloor \right\rfloor$ simple matrices on $|X|$ rows and $ \lfloor c |X|^{k-1}\rfloor $ columns. By our assumption, each $B_i$ makes at least $|X|^{1-\gamma}$ contributions, so in total $M_X$ (and hence $M$) makes at least 
\begin{eqnarray*} 
\left( \frac{c''m^{k-1}}{c|X|^{k-1}}-1\right)|X|^{1-\gamma} &\geq& \frac{c''}{2c}\frac{m^{k-1}}{|X|^{k-1}}|X|^{1-\gamma} \\
  &\geq& k^{k+\gamma - 2}|X|^{-k-\gamma+2}m^{k-1} \\
  &\geq& m^{k-1-(1-\gamma')(k+\gamma-2)} \\
&=& m^{1-\gamma'} 
\end{eqnarray*}
contributions, by our choice of $c''$ and $\gamma'$. 
\end{proof}

\newtheorem{2FS}[compressions]{Theorem}
\begin{2FS}
For any $\epsilon >0$ there exists a constant $c>0$ such that, for all $m\geq k$, every simple $m \times \lfloor cm^{k-1}\rfloor$ matrix makes at least $m^{1-\epsilon}$ contribution. 
\label{2FS}
\end{2FS}
\begin{proof}
Notice that by Proposition \ref{shattered}, the condition of Lemma \ref{mstep} is true for $\gamma= 1$. By repeatedly applying this lemma, it is easily seen that for any $n \in \N$ there exists $c_n$ such that an $m \times \lfloor c_n m^{k-1} \rfloor$ simple matrix makes at least $m^{1-\gamma_n}$ contributions where $$\gamma_{n+1} = \frac{\gamma_n}{k-1+\gamma_n}$$ and $\gamma_0 = 1$. Now the sequence $(\gamma_n)$ is decreasing and bounded below by $0$ so it must converge to a limit $\gamma$ as $n\to \infty$, and since $\gamma = \gamma /(k-1+\gamma)$ we have $\gamma = 0$.  So by choosing $n$ large enough we may ensure that any $m\times \lfloor c_n m^{k-1} \rfloor$ simple matrix makes  $m^{1-\epsilon}$ contributions. 
\end{proof}
\section{$k \times l$ matrices} 

For the general case, we need to be more efficient in finding shattered $k$-sets when considering the family $\tA$ given by Corollary \ref{isetsplus}. The next lemma will be useful in doing so. 
\newtheorem{sstep}{Lemma}[section]
\begin{sstep}
Let $i,k \in \N$, with $1 \leq i \leq k$ and let $0< \gamma \leq k-1 $. Suppose that there exists a constant $c >0$ such that every simple $m\times \lfloor cm^{k-1} \rfloor$ matrix  makes $m^{k-1 - \gamma}$ contributions. Let $M$ be a simple matrix, $\A$ its associated family and $\tA$ the family given by Corollary \ref{isetsplus}. Suppose that $M$ does not make $m^{k-1 - \gamma'}$ contributions, where $0< \gamma' \leq k-1$. For any $b,d \geq 0$, there exists a constant $d' = d'(k,c,b,d)$ such that if $\tA' \subseteq \tA$ is a subfamily of $\tA$ of size at least $d'm^{k-i}$ then there exists $X \subseteq [m]$ such that 
\begin{enumerate}
\item $|X| \geq bm^{(\gamma' - i  + 1)/\gamma}$ and
\item each $x \in X$ belongs to at least $dm^{k-i-1}$ elements of $\tA'$. 
\end{enumerate}
\label{sstep}
\end{sstep}
\begin{proof}
Let $\tA'$ be a subfamily of $\tA$ of size at least $d'm^{k-i}$ where $$ d' =\max\{4cb^{k-1} , 2d\} .$$  Consider the finite sequences $(X_j)_{j=0}^{r}$, $(R_j)_{j=0}^{r}$ and $(\mathcal{F}_j)_{j=0}^{r}$ generated by the following algorithm:

Begin with $X_0 = [m]$,  $\mathcal{F}_0 = \tA'$, $R_0 = \emptyset$. Thereafter,
\begin{itemize}
\item
if $|\mathcal{F}_j| < \frac{d'}{2}m^{k-i}$ then STOP. 
\item
otherwise 

\begin{itemize}
\item let $X_{j+1} = \supp(\F_j)$,
\item 
let $R_{j+1}$ be a subset of $\mathcal{F}_j$ of minimal size such that for any $x \in X_{j+1}$ there exists $A \in R_{j+1}$ with $x\in A$,
\item
let $\F_{j+1}=  \F_j \backslash R_{j+1}$.
\end{itemize}
\end{itemize}
Clearly the algorithm does eventually stop because $|\F_{j+1}| < |\F_j|$. We take $X$ to be $X_r$; let us check that it satisfies the conclusion of the lemma. 

First we show that $X$ has the correct size. Consider a simple submatrix $M'$ of $M$ 
whose rows are indexed by $X$. As $|\F_r| \geq \frac{d'}{2}m^{k-i}$ there exists such a 
matrix on at least $\frac{d'}{2}m^{k-i}$ columns (this is due to Corollary \ref{isetsplus}, 
given that  $\F_r$ is a subfamily of $\tA$). $M'$ contains the concatenation of at least 
$\frac{d'}{2}m^{k-i}/(c|X|^{k-1}+1)$ simple matrices on $m$ rows and at least 
$c|X|^{k-1}$ columns each. By assumption, each of these matrices makes at least $|X|^{k-1- 
\gamma}$ contributions. However, $M$ does not make $m^{k-1 - \gamma'} $ contributions. 
Therefore, one has 

$$\frac{\frac{d'}{2}m^{k-i}}{2c|X|^{k-1}}|X|^{k-1-\gamma} \leq m^{k-1 - \gamma'}, $$ 
and so
$$|X| \geq \left( \frac{d'}{4c}\right)^{\frac{1}{\gamma}}m^{(\gamma' - i  + 1)/\gamma}.$$
As $\gamma \leq k-1$ this implies

$$|X| \geq \left( \frac{d'}{4c}\right)^{\frac{1}{k- 1}}m^{(\gamma' - i  + 1)/\gamma} $$

as required, since $d'$ is large enough that $\left( \frac{d'}{4c}\right)^{\frac{1}{k- 1}}\geq b$.

It remains to show that the second requirement on $X$ holds. As each $X_j$ is a subset of $[m]$, $|R_j| \leq m$. Thus $|\F_j| \geq |\F_0| - jm$. Since $|\F_0| \geq d'm^{k-i}$, the algorithm does not stop in fewer than $\frac{d'}{2}m^{k-b-1}$ steps; equivalently $r\geq \frac{d'}{2}m^{k-i-1}$. Since $X_0 \supseteq X_1 \supseteq \cdots \supseteq X_r$, each $x \in X$ belongs to an element of $R_j$ for $j=0,1,\ldots,r$. As these are pairwise disjoint subsets of $\tA'$, this finishes the proof of the lemma, since $d' \geq 2d$.  
\end{proof}

We denote by $f(k,c,b,d)$ the minimal value of $d'$, as a function of $k$, $c$, $b$ and $d$, which guarantees the existence of the set $X$ in Lemma \ref{sstep}. We now prove an equivalent of Lemma \ref{mstep} for general $k$. 

\newtheorem{bstep}[sstep]{Lemma}
\begin{bstep}
Let $k \in \N$. Let $0 \leq \gamma \leq k-1$ be a real number. Suppose that there exists a constant $c >0$ such that every simple $m\times  \lfloor cm^{k-1} \rfloor$ matrix  makes  $m^{k-1 - \gamma}$ contributions. Then there exists a constant $c'$ such that, for all $m\geq k$, every simple $m\times \lfloor c'm^{k-1} \rfloor $ matrix makes  $m^{k-1-\gamma'}$ contributions, where $$\gamma' = \frac{-2\gamma - 1+ \sqrt{(2\gamma+1)^2 + 8\gamma(k-1)}}{2}.$$
\label{bstep}
\end{bstep}
\begin{proof}
For notational simplicity, let $a = \lfloor \gamma' + 1 \rfloor$ and define $$c_1 = f(k,c,a!{k  \choose a},1)$$ and for $2\leq i \leq a$   $$c_i = f(k,c,1,c_{i-1}).$$ 

It is straightforward to check that $\gamma \in [0,k-1]$ implies $\gamma' \in [0,k-1]$ which in turn implies that $a \in [1,k]$. All the sums of the form $\sum_{i=1}^{j} \frac{\gamma' - j + 1}{\gamma} $ below are therefore well-defined and the applications of Lemma \ref{sstep} below are justified.

Let $M$ be an $m\times \lfloor c'm^{k-1} \rfloor$ simple matrix where $$ c' = 2+c_a.$$ Suppose for a contradiction that $M$ does not make $m^{k-1-\gamma'}$ contributions. Let $\A$ denote the family of subsets of $[m]$ associated with $M$ and let $\tA$ be the family given by Corollary \ref{isetsplus} applied to $\A$.

Given $S \subseteq [m]$ we denote by $\tA^{(\geq k)}_S$ the subfamily of $\tA^{(\geq k)}$ consisting of those sets containing $S$. 
For $1 \leq i \leq a$ we call $(x_j)_{j=1}^i$ a \emph{good} sequence if $x_j \in [m]$ for all $1\leq j \leq i$ and  

$$|\tA^{(\geq k)}_{\{x_1,x_2,\ldots,x_i \}}| \geq c_{a-i}m^{k-i-1}$$ 

We make the following claim: there are at least  $m^{\sum_{j=1}^{i}\frac{\gamma' - j + 1}{\gamma}} $ good sequences for all $1 \leq i \leq a-1$ and at least $a!{ k \choose a} m^{\sum_{j=1}^{i}(1 - 
\frac{\gamma - \gamma' +j - 1}{\gamma})}  $  good sequences when $i=a$. 

We prove the claim by induction on $i$. For $i=1$, this is Lemma \ref{sstep} (with $i=1$, $b = c_{a}$ and $d = 1$) applied to 
$\tA^{(\geq k)}$: the good sequences are the elements of the set $X$ given by the 
lemma.  Suppose now that the result holds for $i <a$. Let $(x_j)_{j=1}^i$ be a good 
sequence.  Apply Lemma \ref{sstep} (with $i$, $b = c_{a-i+1}$ and $d = 1$) to $$\{A \backslash \{x_1,x_2,\ldots,x_i\}: \, A \in 
\tA^{(\geq k)}_{\{x_1,x_2,\ldots,x_i \}}  \} $$ which is a subfamily of $\tA$ (recall 
the latter is a down family). For any $x_{i+1}$ chosen in the set $X$ so obtained, 
$(x_j)_{j=1}^{i+1}$ is a good sequence since a set in $\{A \backslash \{x_1,x_2,\ldots,x_i\}: \, A 
\in \tA^{(\geq k)}_{\{x_1,x_2,\ldots,x_i \}}  \}$ containing $x_{i+1}$ extends uniquely 
to a set of $\tA^{(\geq k)}$ containing $\{x_1,x_2,\ldots,x_{i+1} \}$. We have $m^{(\gamma' - i  + 1)/\gamma}$ choices for $x_{i+1}$ when $i\leq a-1$ and $a!{k \choose a}m^{\frac{\gamma'-a+1}{\gamma}}$ when $i=a$, hence the claim.  

Good sequences provide a lower bound for the number of sets of size $k$ in $\tA$. Let us count the set 
$$Z 
= \{ (Z_1,Z_2) \in \tA^{(a)} \times \tA^{(k)} : Z_1 \subseteq Z_2 
\}  $$
in two different ways. On the one hand, it is clear that $|Z| \leq {k \choose a} |\tA^{(k)}|$. On the other hand, each good sequence consists of elements belonging to a set of size $a$ contained in at least one element of $\tA^{(\geq k)}$, and hence in at least one element of $\tA^{(k)}$ using the fact that $\tA$ is a down family. No more than $a!$ good sequences define the same set of size $a$, hence $|Z| \geq { k \choose a} m^{\sum_{j=1}^{a}\frac{\gamma' - j + 1}{\gamma}}$. Combining the two inequalities on $|Z|$ gives $|\tA^{(k)}| \geq  m^{\sum_{j=1}^{a}\frac{\gamma' - j + 1}{\gamma}}$. Now each set of size $k$ in $\tA$ translates into one contribution in $M$ (again by Corollary \ref{isetsplus}), hence $M$ makes at least $$ m^{\sum_{j=1}^{a}\frac{\gamma' - j + 1}{\gamma}}$$ contributions. 
Notice that 
\begin{equation} 
\label{simplification}
\sum_{j=1}^{\lfloor\gamma'+1\rfloor}\frac{\gamma' - j + 1}{\gamma} \geq \frac{\gamma'(\gamma'+1)}{2\gamma}.
\end{equation}
(this can be justified by seeing that $\sum_{j=1}^{\lfloor\gamma'+1\rfloor}(\gamma' - j + 1)/\gamma = f(\lfloor\gamma'+1\rfloor)$ where $f(x) = x(\gamma'+1)/\gamma - x(x+1)/(2\gamma)$; thus $f$ is a polynomial in $x$ of degree 2 attaining a maximum at $\gamma' + 1/2$ hence $f(\lfloor\gamma'+1\rfloor) \geq f(\gamma')$ which is the  inequality above). Hence $M$ makes at least $m^{\gamma'(\gamma'+1)/(2\gamma)}$ contributions. The solution lying in $[0,k-1]$ to the second degree equation (in $x$) $$k-1-x =  \frac{x(x+1)}{2\gamma} $$ is precisely $\gamma'$. This means that $M$ makes at least $m^{k-1-\gamma'}$ contributions and we have reached a contradiction.
\end{proof}

Applying Lemma \ref{bstep} repeatedly gives the following theorem. 

\newtheorem{preFS}[sstep]{Theorem}
\begin{preFS}
Let $k \in \N$, with $k\geq 3$. For any $\epsilon >0$ there exists a constant $c> 0$ such that every $m\times \lfloor cm^{k-1} \rfloor$ simple matrix makes at least $m^{k-1-\alpha-\epsilon}$ contributions, where $$\alpha = \frac{2k}{3} -1.$$
\label{preFS}
\end{preFS}
\begin{proof}
Notice that by Proposition \ref{shattered}, the condition of Lemma \ref{bstep} is true for $\gamma = k-1$. By repeatedly applying this lemma, it is easily seen that for any $n\in \N$ there exists $c_n$ such that an $m\times \lfloor c_n m^{k-1} \rfloor$ simple matrix makes $m^{k-1-\gamma_n}$ contributions, where $$\gamma_{n+1} =  \frac{-2\gamma_n - 1 + \sqrt{(2\gamma_n+1)^2 +8\gamma_n(k-1)}}{2}$$ and $\gamma_0 = k-1$. We will now show that $\gamma_n$ tends to $\alpha$ as $n\to \infty$. In order to do so it is convenient to consider the real-valued function $h$ defined on $[\alpha, k-1]$ by $$ h(x) = \frac{-2x - 1 + \sqrt{(2x+1)^2 + 8x(k-1)}}{2}.$$ It is straightforward to check that 
\begin{itemize}
\item
$h$ is strictly increasing
\item
$h(x) < x$ for $x\in\, (\alpha, k-1]$ and $h(\alpha) = \alpha.$
\end{itemize}
This implies that $(\gamma_n)$ is a decreasing sequence tending to $\alpha$ as $n \to \infty$, and we are done as in the proof of Theorem \ref{2FS}. 
\end{proof}

\section{Proof of Theorem \ref{maintheorem}}
\begin{proof}[\nopunct]
The first part follows from applying Theorem \ref{2FS} to Lemma \ref{basiclemma}. The second part follows from applying Theorem \ref{preFS} to Lemma \ref{basiclemma}. 
\end{proof}
\section{Concluding remark}
In the proof of Lemma \ref{bstep} we made the simplification (\ref{simplification}). While this does not affect the asymptotic for $\fs(m,F)$ stated in Theorem \ref{maintheorem} for large $k$, it can lead to a small overestimation for some values of $k$. For example, a simple computer program finding the best possible value of $\gamma'$ at each application of Lemma \ref{bstep} suggests that one may prove by hand for $k = 4$ that $\fs(m,F) = O(m^{5.618+o(1)})$, while the bound of Theorem \ref{maintheorem} is roughly $m^{5.6667 +o(1)}$. Likewise for $k = 5$ one should be able to show that $\fs(m,F) = O(m^{7.3028})$ while the stated bound is $\fs(m,F) = O(m^{7.3333+o(1)})$, and so on for larger values of $k$. 

\section*{Acknowledgement}
The author wishes to thank Andrew Thomason for his invaluable comments and suggestions regarding the presentation of this paper.


\begin{thebibliography}{20}
\bibitem{ffp}
P. Frankl, Z. F\"{u}redi and J. Pach, Bounding one-way differences, \emph{Graphs and Combinatorics} \textbf{3}, 341-347.
\bibitem{ansteefuredi}
R. P. Anstee, Z. F\"{u}redi, Forbidden submatrices, \emph{Discrete Mathematics} \textbf{62} (1986), 225-243.
\bibitem{sauer}
N. Sauer, On the density of families of sets, \emph{Journal of Combinatorial Theory Series A} \textbf{40} (1985), 108-124.
\bibitem{shelah}
S. Shelah, A combinatorial problem; stability and order for models, \emph{Journal of Combinatorial Theory Series A} \textbf{13} (1972), 145-147.
\bibitem{ansteebound}
R. P. Anstee, On a conjecture concerning forbidden submatrices, \emph{Journal of combinatorial mathematics and combinatorial computing} \textbf{32} (2000), 185-192.
\bibitem{ansteechen}
R. P. Anstee, Ruiyuan Chen, Forbidden submatrices: some new bounds and constructions, \emph{The electronic journal of combinatorics} \textbf{20(1)} (2013), \#P5.

\end{thebibliography}
\end{document}